\documentclass[12pt, reqno]{amsart}
\usepackage{amsmath, amsthm, amscd, amsfonts, amssymb, graphicx, color}
\usepackage[bookmarksnumbered, colorlinks, plainpages]{hyperref}

\newtheorem{theorem}{Theorem}[section]
\newtheorem{lemma}[theorem]{Lemma}
\newtheorem{proposition}[theorem]{Proposition}
\newtheorem{corollary}[theorem]{Corollary}
\theoremstyle{definition}
\newtheorem{definition}[theorem]{Definition}
\newtheorem{example}[theorem]{Example}

\theoremstyle{remark}
\newtheorem{remark}[theorem]{Remark}
\numberwithin{equation}{section}

\begin{document}
\setcounter{page}{1}

\title[Some new classes of rings which have the McCoy condition]{Some new classes of rings which have \\ the McCoy condition}

\author[P. Danchev]{P. Danchev}
\author[M. Zahiri]{M. Zahiri}

\address{Institute of Mathematics and Informatics, Bulgarian Academy of Sciences, 1113 Sofia, Bulgaria}
\email{\textcolor[rgb]{0.00,0.00,0.84}{pvdanchev@yahoo.com; danchev@math.bas.bg}}
\address{Department of Pure Mathematics, Faculty of Mathematical Sciences, Tarbiat Modares University, 14115-134 Tehran, Iran}
\email{\textcolor[rgb]{0.00,0.00,0.84}{m.zahiri86@gmail.com}}

\thanks{*Corresponding author: m.zahiri86@gmail.com}


\begin{abstract} We define here the notion of a {\it weakly reversible ring} $R$ saying that a non-zero element $a\in R$ is weakly reversible if there exists an integer $m>0$ depending on $a$ such that $a^m\neq 0$ is reversible, that is, $r_R(a^m)=l_R(a^m)$. In addition, $R$ is weakly reversible if all its elements are weakly reversible.

It is shown that all weakly reversible rings are abelian McCoy rings and so, particularly, they are abelian 2-primal rings. Moreover, we construct a weakly reversible ring which is {\it not} reversible.

We also show that if $R$ is a weakly reversible ring, then the polynomial ring $R[x]$ is strongly AB. Thus, in particular, the weakly reversible ring $R$ is zip if, and only if, $R[x]$ is zip. We, moreover, prove that if $R$ is a weakly reversible ring and every prime ideal of $R$ is maximal, then both $R$ and $R[x]$ are AB rings.
\end{abstract}

\maketitle

\section{Introduction and Motivation}

Throughout this article, all rings are associative with identity. Almost all notions and notations are standard being in agreement with the well-known classical books \cite{Lam2} and \cite{Lam}. As usual, we denote by $W(R)$ the sum of the nilpotent ideals of $R$ (called the {\it Wedderburn radical} of $R$), by $Nil_{*}(R)$ the intersection of all prime ideals of $R$ (called the {\it lower nil-radical} of $R$), by $L$-$rad(R)$ the largest locally nilpotent ideal
of $R$ (called the {\it Levitzky radical} of $R$), by $Nil^{*}(R)$ the sum of all nil-ideals (called the {\it upper nil-radical} of $R$), and by $Nil(R)$ the set of all nilpotent elements of $R$.

We always have the inclusions:

$$W(R)\subseteq N il_{*}(R) \subseteq Nil^{*}(R)\subseteq Nil(R).$$

\medskip

\noindent Likewise, we recall that a ring $R$ is {\it $2$-primal} if $Nil_{*}(R) = Nil(R)$.

In the other vein, a ring $R$ is said to be {\it reversible} if the condition $ab=0$ implies that $ba=0$ for all elements $a,b\in R$. Mimicking Feller \cite{fel}, a ring $R$ is said to be a {\it right} (resp., {\it left}) {\it duo}, provided every right (resp., left) ideal of it is an ideal.

Besides, in \cite{zah}, a ring $R$ is said to be {\it nil-reversible}, if $ab=0$ implies $ba=0$ for all $a\in Nil(R)$ and $b\in R$. Clearly, any reversible ring is nil-reversible, but the converse implication manifestly fails.

Imitating \cite{moh}, a ring $R$ is said to be {\it left $\pi$-duo} if, for every $0\neq  a \in R$, there exists a positive integer $k$ for which $0\neq a^kR \subseteq Ra^k$. The definition of a right $\pi$-duo is quite similar. If a ring is both left and right $\pi$-duo, it is just called $\pi$-duo. Recall also that a ring $R$ is {\it $\pi$-CN}   whenever every nilpotent element in $R$ of index $2$ is central.

Referring to McCoy \cite[Theorem 2]{NHM}, all commutative rings have the property that, if the annihilator of a polynomial in $R[x]$ is a non-zero, then the annihilator coefficient of the polynomial is non-zero in $R$ too.  However, this property may not be longer true in noncommutative rings. Nevertheless, it plays an important role in noncommutative algebras, because the zip property of noncommutative rings may not pass to their polynomial rings. 

Furthermore, in \cite{niel}, Nielsen introduced the {\it McCoy property} of rings, becoming the most widely used tool for studying the zero-divisors and annihilators of a ring extension: a ring $R$ is said to be {\it right McCoy} if the   condition $f(x)g(x)=0$ implies the condition $f(x)c=0$ for some non-zero $c\in R$, where $f(x),g(x)$ are non-zero polynomials in $R[x]$. Additionally, {\it left McCoy} rings are defined dually. Generally, a ring $R$ is called {\it McCoy} if it is both left and right McCoy. This is really a very large class of rings as the next specifications illustrate: Nielsen proved there that all reversible rings are themselves McCoy. In this aspect, in \cite{Camillo}, the authors show that right duo rings are just right McCoy. Also, in \cite{zah}, it was proven that nil-reversible rings are McCoy as well. Following \cite{moh}, any $\pi$-CN ring and any $\pi$-duo ring are McCoy too.  \\

We now demonstrate the following irreversible relationships: \\

\indent\indent\indent\indent\indent\indent\indent\indent\indent\indent\indent\indent right duo\,\,\, $\longrightarrow$ $\pi$-duo\\
\indent\indent \indent\indent\indent\indent\indent
\indent\indent\indent$\nearrow$\indent\indent\indent\indent\indent\indent\indent \indent\indent\indent \indent$\searrow $\\
\indent\indent\indent\indent  $\emph{commutative}\rightarrow\emph{reversible}\rightarrow \emph{nil-reversible}\rightarrow$  McCoy       \\
\indent\indent\indent\indent\indent\indent\indent\indent\indent\indent$\searrow$\indent\indent\indent\indent\indent\indent \indent\indent\indent\indent\indent$\nearrow$

\indent \indent\indent\indent\indent\indent\indent \indent\indent\indent\indent\indent$\ {CN}$\ \ \ $\longrightarrow \indent \text{$\pi$-\textit{CN}}$\\     \indent\indent\indent \indent \indent\indent\indent\indent\indent\indent\indent  \indent\indent\indent\indent\indent \indent\indent\,\,\,\,  $\downarrow \\      \indent\indent\indent \indent \indent\indent\indent\indent\indent\indent\indent   \indent\indent\indent\indent\indent\indent\,\,\,\text{2-primal} $\\

We shall show in the sequel that the condition "reversibility" can, actually, be slightly expanded to the so-termed "weak reversibility" so that the latter is still satisfying the McCoy property.

\medskip

The concept of bounding an one-sided ideal by a two-sided ideal goes back at least to Jacobson \cite{11}. He stated that the right ideal of a ring $R$ is {\it bounded}, provided that it contains a non-zero ideal of $R$. In several ways, this terminology has been extended as follows. According to Faith \cite{5}, a ring $R$ is called {\it strongly right} (resp., {\it strongly left}) {\it bounded} if every non-zero right (left) ideal of $R$ contains a non-zero ideal. So, a ring is called {\it strongly bounded} if it is simultaneously strongly right and strongly left bounded. It is principally known that right (resp., left) duo rings are strongly right (resp., left) bounded and semi-commutative.

On the other hand, a ring $R$ is said to be right (resp., left) {\it AB} if every essential right (resp., left) annihilator of $R$ is bounded. Moreover, owing to \cite{Hwang}, a ring $R$ is called {\it strongly right} (resp., {\it strongly left}) {\it AB} if every non-zero right (resp., left) annihilator of $R$ is bounded; thereby, $R$ is called {\it strongly AB} if $R$ is simultaneously strongly right and strongly left AB. Obviously, strongly right bounded
rings and semi-commutative rings are both strongly right AB.

In the present paper, we arrive at the curious fact that the lower radical coincides with the set of nilpotent elements in a weakly reversible ring, as defined in the next section, so that any weakly reversible ring $R$ is necessarily 2-primal, and so the equality $Nil(R[x])=Nil(R)[x]$ holds (see, for instance, \cite{Lam2} and \cite{Lam}). Also, we will see that all weakly reversible rings are abelian, i.e., their idempotent elements are central (see Proposition~\ref{abel}).

Our basic motivating tool, however, is the existence of a weakly reversible ring that is {\it not} reversible (see Example~\ref{major}). In this way, our main result states that weakly reversible rings are McCoy (see Corollary~\ref{mccoy}), which is a direct consequence of the more general setting that the polynomial ring $R[x]$ is always strongly AB over the weakly reversible ring $R$ (see Theorem~\ref{maj}). As a consequence, we show that when $R$ is a right zip weakly reversible ring, then $R[x]$ is right strongly AB (see Corollary~\ref{zip}).

\section{Right and left weakly reversible rings}

We begin our work with a pivotal instrument which slightly expands the property of being reversible.

\begin{definition} We say that an element $a\in R$ is {\it reversible} if, for each $b\in R$, $b\in l_R(a)\Longleftrightarrow b\in r_R(a)$. Thus, a ring $R$ is {\it reversible} if all elements of $R$ are reversible.
We call a ring $R$ {\it weakly reversible} if, for each non-zero element $a\in R$, there is an integer $m>0$ such that $a^m$ is a non-zero reversible element of $R$.
\end{definition}

Traditionally, for every ring $R$ and $a\in R$, we denote by $RaR$ the two-sided ideal of $R$ generated by $a$.

\medskip

We now ready to exhibit our crucial construction of the existence of a weakly reversible ring which is {\it not} reversible.

\begin{example}\label{major} There is a weakly reversible ring that is {\it not} reversible.

In fact, let $F$ be a field and set $A := F\langle x, y\rangle$, where $x$ and $y$ are non-commuting indeterminates. Suppose also that $I$ is the two-sided ideal $AxyA+Ay^2xA+Ayx^2A+Ax^3A+Ay^3A$ of the ring $A$. Note that
every element of the factor-ring $R:=A/I$ can be written uniquely in the form $$\alpha=a+\sum_{i=1}^2a_ix^i+\sum_{j=1}^2b_jy^j+cyx,\ \text{ where } a,a_i,b_j,c\in F.$$
It not so hard to see that $$Nil(R)=Fyx+Fx+Fx^2+Fy+Fy^2$$ and that $$\alpha=s+\sum_{i=1}^2r_ix^i+\sum_{j=1}^2t_jy^j+dyx\in R$$ is invertible precisely when $s\neq 0$, where $s,r_i,t_i,d\in F$. This unambiguously shows that the set of all left and right zero divisor elements of $R$ is contained in $Nil(R)$.

Notice that $xy=0$ and $yx\neq 0$ imply that $R$ is {\it not} a reversible ring. However, we next show that $R$ is a weakly reversible ring. To that end, consider the non-zero element $$\alpha=\sum_{i=1}^2r_ix^i+\sum_{j=1}^2t_jy^j+dyx\in R.$$ Then, we differ two basic cases:

\medskip

$\textbf{Case 1}$ -- $r_1=t_1=0$: So, $\alpha=r_2x^2+t_2y^2+dyx$ and it is evident in this case that $\alpha$ is a reversible element of $R$.

\medskip

$\textbf{Case 2}$ -- Either $r_1 \neq 0$ or $t_1\neq 0$: Thus, $\alpha^{2}=r_1^2x^2+t_1^2y^2+t_1r_1yx$ is a non-zero reversible element of $R$, and hence $R$ is a weakly reversible ring, as promised.
\end{example}

Given a ring $R$, we denote $S_n(R)$ to be the ring $$ \left\{ \left. \begin{pmatrix}
 a_{0} & a_{12} & a_{13} & \cdots & a_{1n} \\
 0 & a_{0} & a_{23} & \cdots & a_{2n} \\
 0 & 0 & a_{0} & \cdots & a_{3n} \\
 \ddots & \ddots & \ddots & \vdots & \ddots \\
 0 & 0 & 0 & \cdots & a_{0}
\end{pmatrix} \right| a_0, a_{ij}\in R \right\}.$$

We are now in a position to establish the following surprising criterion, which gives some transversal between the properties "reversible" and "weakly reversible".

\begin{lemma}
Let $R$ be a ring. Then, $S_2(R)$ is weakly reversible if, and only if, $R$ is reversible.
\end{lemma}

\begin{proof} Assume that $R_2(R)$ is weakly reversible. If $ab=0$, then $$\begin{pmatrix}
 0 & a  \\
 0 & 0\\
\end{pmatrix} \begin{pmatrix}
b& 0  \\
 0 & b\\
\end{pmatrix}=0.$$ As a simple check gives that $\begin{pmatrix}
 0 & a  \\
 0 & 0\\
\end{pmatrix}^2=0$, we deduce $\begin{pmatrix}
 0 & a  \\
 0 & 0\\
\end{pmatrix} $ is a reversible element of $S_2(R)$, which leads to $$\begin{pmatrix}
 b & 0  \\
 0 & b\\
\end{pmatrix} \begin{pmatrix}
0& a  \\
 0 & 0\\
\end{pmatrix}=0.$$ This means that $ba=0$, i.e., $R$ is indeed reversible.\\
Conversely, assume that $R$ is a reversible ring.
Letting $A=\begin{pmatrix}
a& b  \\
 0 & a\\
\end{pmatrix}$ and $B=\begin{pmatrix}
c& d  \\
 0 & c\\
\end{pmatrix}\in R_2(R)$ with $AB=0$, we derive that $ac=0$ and $ad+bc=0$. From $ac=0$, we get $ca=0$ whence $cab=cba=0$. But, multiplying $ad+bc=0$ on the left side by $a$, we receive $a^2d+abc=0$. As $ac=0$, we get $abc=0$. Finally, $a^2d=0$ which forces that $da^2=0$. Now, it is easy to see that $BA^2=0$. Consequently, $S_2(R)$ is weakly reversible, as asserted.
\end{proof}

Two more technicalities are also valid.

\begin{lemma} Let $R$ be a ring. Then, $S_n(R)$ is not weakly reversible for any $n\geq 3$.
\end{lemma}

\begin{proof} Put $A:=e_{12}$ and $B:=e_{23}$. Then, one readily inspects that $A^2=B^2=0$ and $BA=0$, but $AB\neq 0$, as required.
\end{proof}

The following is now immediate.

\begin{lemma}\label{4} Let $R$ be a weakly reversible ring. Then, $0\neq a\in R$ is reversible whenever $a^2=0$.
\end{lemma}

The next statement is key.

\begin{proposition}\label{abel} Every weakly reversible ring is abelian.
\end{proposition}

\begin{proof} Assume that $e=e^2\in R$. Then, it plainly follows that $(eR(1-e))^2=0$ yielding $eR(1-e)$ is a reversible subset of $R$. Hence, $$eR(1-e) = e^{2}R(1-e) = e(eR(1-e))=eR(1-e)e =0.$$ This allows us that $eR = eRe$. Similarly, from $((1-e)Re)^{2} = 0$, we detect $(1-e)Re=0$ and so $Re=eRe$. Hence $e$ is a central idempotent, as wanted.
\end{proof}

As a direct consequence, we record the following.

\begin{corollary} If $e\in R$ is an arbitrary idempotent and $R$ is a weakly reversible ring, then the corner subring $eRe$ is too weakly reversible.
\end{corollary}

\begin{proof} The previous proposition stands that $eRe=eR=Re$. Now, if $eae$ and $ebe$ are both in $R$ such that $eaeebe=0$ for some $a\in R$, then there exists an integer $k$ such that $(eae)^k$ is a reversible element in $R$, whence $(eae)^k$ is reversible in $eRe$, as required.
\end{proof}

However, the converse implication does {\it not} hold in general: in fact, if both $eRe$ and $(1-e)R(1-e)$ are weakly reversible, then $R$ need {\it not} be so. Indeed, choose $R=M_2(\mathbb{Z}_2)$, and let $e=E_{11}$. An easy check shows that both $eRe$ and $(1-e)R(1-e)$ are isomorphic to $\mathbb{Z}_2$, so that they are obviously weakly reversible. But, it is easily verified that $R$ is non-abelian and, therefore, Proposition~\ref{abel} teaches us that $R$ cannot be weakly reversible, as claimed. Besides, one observes that some elements of $R$, such as $E_{12}$, are not $n$-potent for any $n$.

\medskip

We now proceed by proving the following necessary and sufficient condition.

\begin{proposition} A weakly reversible ring is non-singular if, and only if, it is reduced.
\end{proposition}

\begin{proof} Apparently, any reduced ring is non-singular.

For the converse, assume on the contrary that $R$ is non-singular and $a^2=0\neq a$. As $R$ is weakly reversible, Lemma~\ref{4} tells us that $a$ is a reversible element of $R$. Now, consider $0\neq b\in R$. If $ab=0$, then $bR\cap r_R(a)\neq \{0\}$.\\ If, however, $ab\neq 0$, then $ba\neq 0$, because $a$ is a reversible element from $R$. Since $a^2=0$, we infer $Ra^2=0$ and so $aRa=0$, because of the reversibility of $a$ in $R$. As $ba\subseteq Ra$, we find that $aba=0$, which insures that $0\neq ba\in r_R(a)$. This guarantees that $bR\cap r_R(a)\neq \{0\}$. Thus, $r_R(R)$ is an essential right ideal of $R$ being a contradiction to the non-singularity of $R$. Finally, $a=0$ ensuring that $R$ is reduced, as needed.
\end{proof}

Our further work continues with a series of useful preliminaries.

\begin{lemma} Let $R$ be a weakly reversible ring and $a,b\in R\backslash\{0\}$ with $ab=0$. Then, there exists an integer $t$ such that $a^tRb =b Ra^t=\{0\}$ with $a^t\neq0$.
\end{lemma}

\begin{proof} Suppose that $a,b\in R\backslash\{0\}$ with $ab=0$. As $R$ is a weakly reversible ring, then there exists an integer $t$ such that $a^t\neq0 $ and $a^t$ is a reversible element of $R$. But since $a^tbR=0$, we have $bRa^t=\{0\}$. On the other side, $Rba^t=\{0\}$ assures that $a^tRb=\{0\}$, as pursued.
\end{proof}

\begin{lemma}\label{5} Let $R$ be a weakly reversible ring and $a,b\in R\backslash\{0\}$ with $ab=0$. Then, there exists an integer $k$ such that $aRb^k=b^kRa=\{0\}$ with $b^k\neq0$.
\end{lemma}

\begin{proof} Suppose that $a,b\in R\backslash\{0\}$ with $ab=0$. As $R$ is a weakly reversible ring, then there exists an integer $k$ such that $b^k\neq0$ and $b^k$ is a reversible element of $R$. Therefore, $b^kRa=\{0\}$. On the other side, $b^kaR=\{0\}$ ensures that $aRb^k=\{0\}$, as asked for.
\end{proof}

The next statement is helpful.

\begin{lemma}\label{nil} Let $R$ be a weakly reversible ring. If $a\in R$ such that $a^m=0$ for some $m>0$, then   $(aR)^m=0$.
\end{lemma}

\begin{proof} Given $0\neq a\in Nil(R)$, then there is an integer $m\geq2$ such that $a^{m-1}\neq0=a^m$. Then, one checks that $Raa^{m-1}=\{0\}$, and since $R$ is a weakly reversible ring and $2m-2\geq m$, we get $a^{2m-2}=0$, so $$a^{m-1}Ra=\{0\}.\indent\indent\indent\indent\indent\indent (1)$$ It readily follows now that $Ra(a^{m-2}Ra)=\{0\}$. Furthermore, employing $(1)$, we obtain that $(a^{m-2}Ra)^2=\{0\}$, so with the weakly reversibility at hand, we get $$(a^{m-2}Ra)Ra=\{0\}.\indent\indent\indent\indent\indent\indent (2)$$ From this, it must be that $Ra(a^{m-3}RaRa)=\{0\}$, and by $(2)$ it can easily be checked that $(a^{m-3}RaRa)^2=\{0\}$. Therefore, using this and the fact that $R$ is weakly reversible, we deduce that $$a^{m-3}RaRaRa=\{0\}.\indent\indent\indent\indent\indent\indent (3)$$Continuing this process finitely many times, we conclude $(aR)^{m}=\{0\}$, as required.
\end{proof}

We are now prepared to prove the following two claims.

\begin{lemma}\label{4new} Let $R$ be a weakly reversible ring and $a,b\in R\backslash\{0\}$ with $ab=0$. Then, there exists an integer $k$ such that $(aR)^kb =\{0\}$ and $b(Ra)^k=\{0\}$ with $a^k\neq0$.
\end{lemma}

\begin{proof} Suppose that $a,b\in R$ are non-zero elements of $R$ such that $ab=0$. We may additionally assume also that $aRb\neq \{0\}$. If $a\in Nil(R)$, then Lemma \ref{nil} applies to conclude that there exists an integer $n$ such that $(RaR)^n=\{0\}\neq (RbaR)^{n-1}$. Then, $(RaR)^{n-1}$ is a reversible subset of $R$. So, $$b(RaR)^{n-1}=\{0\}=(RaR)^{n-1}b,$$ and the argumentation is completed in this case.\\
If, however, $a\not\in Nil(R)$, then, in view of Lemma \ref{4}, there is a positive integer $m\geq1$ minimal with respect to the requirement that $a^mRb=\{0\}$ and $bRa^m=\{0\}$ with $a^m\neq 0$. It automatically follows now that $(a^{m-1}RbRa)^2=\{0\}$. But, since $R$ is a weakly reversible ring and $a(a^{m-1}RbRa)=\{0\}$, we then can write $a^{m-1}RaRaRa=\{0\}$.

In an analogous direction, from $(a^{m-2}RbRaRa)^2=\{0\}$ and $a((a^{m-2}RbRaRa)=\{0\}$, we write that $a^{m-2}RbRaRaRa=\{0\}$. Continuing this process finitely many times, we obtain $(aR)^{m+1}b=\{0\}$ such that $a^{m+1}\neq 0$.\\
Now, from $a^mRb=\{0\}$ and $a^m\neq 0$, we deduce that $(abRa^{m-1})^2=\{0\}$. So, $abRa^{m-1}$ is a reversible element of $R$. It thus follows that $aRabRa^{m-1}=\{0\}$. Again continuing this process finitely many times, we receive $(aR)^{m+1}b=0$ such that $a^{m+1}\neq 0$, as suspected.
\end{proof}

\begin{lemma}\label{2.5} Let $R$ be a weakly reversible ring and $a,b\in R\backslash\{0\}$ with $ab=0$. Then, there exists an integer $k$ such that $a(Rb)^k =\{0\}$ and $(bR)^ka=\{0\}$ with $b^k\neq0$.
\end{lemma}

\begin{proof} Suppose that $a,b\in R$ are non-zero elements of $R$ such that $ab=0$. We may additionally assume also that $aRb\neq \{0\}$. If $b\in Nil(R)$, then Lemma \ref{nil} is applicable to conclude that there exists an integer $n$ such that $(RbR)^n=\{0\}\neq (RbR)^{n-1}$. Then, $(RbR)^{n-1}$ is a reversible subset of $R$. So, $$a(RbR)^{n-1}=\{0\}=(RbR)^{n-1}a,$$ and the arguments are completed in this case.\\
If, however, $b\not\in Nil(R)$, then, in virtue of Lemma \ref{5}, there is a positive integer $m\geq1$ minimal with respect to the restriction that $aRb^m=\{0\}$ and $b^mRa=\{0\}$ with $b^m\neq 0$. It immediately follows now that $(b^{m-1}RaRb)^2=\{0\}$. But, since $R$ is a weakly reversible ring and $b(b^{m-1}RaRb)=0$, we then can write $b^{m-1}RaRbRb=\{0\}$.

In a way of similarity, from $(b^{m-2}RaRbRb)^2=0$ and $b((b^{m-2}RaRbRb)=0$, we write that $b^{m-2}RaRbRbRb=\{0\}$. Continuing this process a finite many times, we obtain $a(Rb)^{m+1}=\{0\}$ such that $b^{m+1}\neq 0$.\\
Now, from $aRb^m=\{0\}$ and $b^m\neq 0$, we detect that $(baRb^{m-1})^2=\{0\}$. So, $baRb^{m-1}$ is a reversible element of $R$. It thus follows that $b^2aRb^{m-1}=\{0\}$. Again continuing this process a finite many times, we receive $(Rb)^{m+1}a=0$ such that $b^{m+1}\neq 0$, as expected.
\end{proof}

The next assertion is somewhat unexpected.

\begin{lemma}\label{nilid} Let $R$ be a weakly reversible ring. Then, $Nil(R)$ is an ideal of $R$.
\end{lemma}

\begin{proof} Applying Lemma \ref{nil}, we arrive at the equality $RNil(R)R=Nil(R)$. So, it remains to establish only that $Nil(R)$ is a subgroup of $R$. To this target, assuming that $a,b\in nil(R)$, then there exist two integers $m,n$ such that $a^m=b^n=0$. Therefore,  $$(a+b)^{m+n+1}=\sum_{i_1+\cdots+i_t+j_1+\cdots+j_t=m+n+1}a^{i_1}b^{j_1}a^{i_2}b^{j_2}\cdots a_{i_t}b^{j_t},$$ where $0\leq i_1,j,\ldots,i_t,j_t\leq m+n+1$.
Now, consider the element $a^{i_1}b^{j_1}a^{i_2}b^{j_2}\cdots a_{i_t}b^{j_t}$. If $i_1+i_2\cdots+i_t< m$, then $j_1+j_2+\cdots+j_t> m+1$ and so one sees that $$b_{j_1+j_2+\cdots+j_t}=b^{j_1}b^{j_2}\cdots b_{j_t}=0.$$ So, invoking  Lemma \ref{nil}, we get $$a^{i_1}b^{j_1}a^{i_2}b^{j_2}\cdots a_{i_t}b^{j_t}=0.$$ If $i_1+i_2\cdots+i_t\geq m$, then it must be that $$a^{i_1}a^{i_2}\cdots a^{i_t}=a^{i_1+i_2\cdots+i_t}=0,$$ which yields by analogy that $$a^{i_1}b^{j_1}a^{i_2}b^{j_2}\cdots a_{i_t}b^{j_t}=0.$$ Hence, $(a+b)^{m+n+1}=0$, and we are done.
\end{proof}

The next comments are worthwhile.

\begin{remark} If $R$ is a weakly reversible ring, then Lemma~\ref{nilid} is a guarantor that $Nil(R)$ is a two-sided ideal of $R$, and so the factor-ring $R/Nil(R)$ is a reduced ring. However, the reciprocal assertion is not true in general although it is pretty clear that any reduced ring is reversible and so obviously weakly reversible: for example, let $F$ be a field and $S=T_n(F)$ be the upper triangular matrix ring -- we know that $Nil(S)$ is an ideal of $S$ as well as that $S/Nil(S)$ is isomorphic to the direct product $F^n$ of exactly $n$ copies of $F$ -- so, the quotient is reduced and hence a weakly reversible ring, but $S$ is manifestly {\it not} weakly reversible ring, as asserted.
\end{remark}

We are now managed to show validity of certain critical properties concerning the weak reversibility of the polynomials.

\begin{lemma} Let $R$ be a weakly reversible ring. If $$f(x)=\sum ^{m} _{i=0} a_{i} x^{i},g(x) = \sum ^{n}_{j =0} b_{j} x ^{j}\in R[x]$$ with $f(x)g(x)=0$, then there exists a positive integer $k$ with $(a_0R)^{k}g(x)=0$.
\end{lemma}

\begin{proof} If $a_0\in Nil(R)$, then there exists an integer $k$ such that $(a_0R)^k=0$. It follows at once that $(a_0R)^kg(x)=0$.\\
If $a_0\not\in Nil(R)$, then Lemma \ref{4} employs to get that there exists an integer $k$ such that $(a_0R)^kb_0=b_0(a_0R)^k=0$.\\ Assume now by induction that $(a_0R)^kb_j=b_j(a_0R)^k=0$ for $0\leq j< n$. From this and $$a_0b_{j+1}+a_1b_j+\cdots+a_jb_1+a_{j+1}b_0=0,$$ we obtain that $$(a_0R)^ka_0b_{j+1}+(a_0R)^ka_1b_j+\cdots+(a_0R)^ka_jb_1+(a_0R)^ka_{j+1}b_0=0.$$ Since $(a_0R)^ka_i\subseteq (a_0R)^k$ for $1\leq i\leq j+1$, we have $(a_0R)^ka_i b_t=0$, $1\leq i\leq j+1$ and $0\leq t\leq j$. Thus, $(a_0R)^ka_0b_{j+1}=0$. It follows that $(b_{j+1}(a_0R)^ka_0)^=0$, and so $b_{j+1}(a_0R)^ka_0$ is a reversible subset of $R$. But, as $R(a_0R)^ka_0b_{j+1}=0$, we write $$R(a_0R)^ka_0\in l_R(b_{j+1}(a_0R)^ka_0)=r_R(b_{j+1}(a_0R)^ka_0).$$ So, $b_{j+1}(a_0R)^ka_0R(a_0R)^ka_0=0$, which enables us that $b_{j+1}(a_0R)^{2k+2}=0$. Therefore, $((a_0R)^{2k+2}b_{j+1})^2=0$, and so $(a_0R)^{2k+2}b_{j+1}$ is a reversible subset of $R$. Since $b_{j+1}(a_0R)^{2k+2}=0$, we get $(a_0R)^{2k+2}b_{j+1}(a_0R)^{2k+2}=0$. It follows that $(a_0R)^{4k+4}b_{j+1}=0$. As $b_{j+1}(a_0R)^{2k+2}=0$, we obtain $b_{j+1}(a_0R)^{4k+4}=0$.\\ Consequently, there exists an integer $t$ such that $(a_0R)^{t}g(x)=0$ and $g(x)(a_0R)^{t}=0$, as desired.
\end{proof}

\begin{lemma}\label{ni} Let $R$ be a weakly reversible ring and $f(x),g(x)\in R[x]$. If $f(x)g(x)=0$, then $a_ib_j\in Nil(R)$, where $f(x)=a_0+a_{1}x+\cdots a_mx^m\in R[x]$ and $g(x)=b_0+b_1x+\cdots +b_nx^{n}\in R[x]$.
\end{lemma}

\begin{proof} We show by induction on $i+j$ that $a_{i}b_{j} \in Nil(R)$, $0\leq i\leq m$ and $0\leq j\leq n$. If $i+j=0$, then $a_{0} b_{0} = 0\in  Nil(R)$ and thus $b_{0}a_{0} \in Nil(R)$. Now, we suppose that our claim is true for $i+j<k$. So, $a_{i} b_{j}\in Nil(R)$ when $i+j<k$. We, therefore, have
	    \begin{equation}\label{eq1}
	    a_{0}b_{k} + a_{1} b_{k-1}  + \cdots + a_{k} b_{0} =0.
	    \end{equation}
But, multiplying from the left side by $b_{0}$, we obtain
	    $$b_{0}a_{0}b_{k} + b_{0}a_{1}\ b_{k-1} + \cdots + b_{0}a_{k} b_{0}=0.$$
Under the presence of the induction hypothesis, it must be that $a_{i} b_{0}\in Nil(R)$ for each $0\leq i< k$. However, since $a_{i}b_{0}\in Nil(R)$, we get $b_{0}a_{i}\in Nil(R)$ for each $0\leq i< k$. Thus, $b_{0}a_{k}b_{0} \in Nil(R)$, because $Nil(R)$ is an ideal of $R$ in conjunction with Lemma~\ref{nilid}. That is why, $b_{0}a_{k}b_{0} \in Nil(R)$ and, therefore, $a_{k}b_{0} \in Nil(R)$.

Furthermore, multiplying equation $(2.1)$ subsequently by $b_{1}, b_{2}\cdots , b_{k-1}$ from left side, we yield that $a_ib_j\in Nil(R)$ for any $0\leq i\leq  k$. Consequently, $a_ib_j\in Nil(R)$, $0\leq i\leq  m$ and $0\leq j\leq  n$, as wanted.
\end{proof}

We now attack the truthfulness of our major result.

\begin{theorem}\label{maj} Let $R$ be a weakly reversible ring. Then, $R[x]$ is simultaneously strongly right AB and strongly left AB (or, in other words, it is a strongly AB ring).
\end{theorem}

\begin{proof} We prove only the right case, as the left one is quite similar. To this goal, suppose $X\subseteq R[x]$ and $r_{R[x]}(X)\neq 0$. Note that the proof is over if we succeed to show $r_R(X)\neq 0$, because if $Xr=0$ for some non-zero $r\in R$, then there exists an integer $k$ such that $r^k$ is a non-zero reversible element of $R$. This allows us that $$r^kX=0 \Longrightarrow r^kXR=0\Longrightarrow XRr^k=0.$$ Therefore, $XR[x]r^k=0$. Hence, we can assume that $Xg(x)=0$ for some $g(x)=b_0+b_1x+\cdots +b_nx^{n}\in S$ having minimal degree such that $b_{n}\neq 0$.

\medskip

We now distinguish two basic cases:

\medskip

\noindent\textbf{Case 1:} If $g(x)c_X=0$, then $c_Xc_{g(x)}$ is a reversible subset of $R$ and so $c_Xc_Xg(x)=0$. If $c_Xg(x)=0$, then $Xb_j=0$, $0\leq j\leq n$. If, however, $c_Xg(x)\neq 0$, then there exists an integer $0\leq t \leq n$ such that $c_Xb_t\neq 0$. Since $c_Xc_Xg(x)=0$, we have $Xc_Xb_t=0$, and the proof ends in this case.

\medskip

\noindent\textbf{Case 2:} If $g(x)c_X\neq 0$, then there exists $c\in c_X$ such that $g(x)c\neq 0$. Utilizing the preceding Lemma \ref{ni}, we have that $g(x)c\in Nil(R[x])=Nil(R)[x]$. But, as $R$ is $2$-primal, we may assume that $b_jcb_ic=0$ for $0\leq i,j\leq n$. So, the set $\{b_jc ~|~ 0\leq j\leq n\}$ is a reversible subset of $R$. Considering now the arbitrary element $f(x)=a_0+a_{1}x+\cdots a_mx^m\in X$, we derive from $f(x)g(x)=0$ that $a_mb_nc=0$, and since $b_nc$ is a reversible element of $R$, we deduce $b_nca_m=0$. Then, $Xg(x)ca_m=0$ and since $deg(g(x)ca_m)< deg(g(x)$, we obtain that $g(x)ca_m=0$. Hence, $a_mg(x)c=0$, because $\{b_jc\}$, $0\leq j\leq n$ is a reversible subset of $R$. From this and $f(x)g(x)c=0$, it follows that $a_{m-1}b_nc=0$. Similarly, we receive $b_nca_{m-1}=0$ and since $g(x)$ has the minimal degree in $r_{R[x]}(X)$, it must be that $g(x)ca_{m-1}=0$ and so $a_{m-1}g(x)c=0$. Continuing this process finitely many times, we get $a_ig(x)c=0$ for $0\leq i\leq m$. As $f(x)$ is arbitrary element of $X$, we finally conclude that $c_Xg(x)c=0$. This, in turn, means that $r_R(X)\neq 0$, as pursued.
\end{proof}

Two immediate valuable consequences are these:

\begin{corollary}\label{mccoy} Let $R$ be a weakly reversible ring. Then, $R$ is a McCoy ring.
\end{corollary}

\begin{corollary}\label{zip} Let $R$ be a weakly reversible ring. Then, $R[x]$ is a zip ring if, and only if, so is $R$.
\end{corollary}

We finish our work with the following challenging question.

\medskip

\noindent{\bf Problem.} Does it follow that weakly reversible $\pi$-duo rings are themselves reversible?

\bigskip

\bibliographystyle{amsplain}

\end{document}